\documentclass[doublecolumn,letterpaper,10pt,conference]{ieeeconf}
\let\labelindent\relax

\usepackage{breakcites}
\usepackage{amsmath,amssymb,amsfonts}
\usepackage{graphicx}
\usepackage{textcomp}
\usepackage{style}
\usepackage{indentfirst}
\usepackage{rotating}
\usepackage{algorithm,algorithmic}
\usepackage{pgf}
\usepackage{tikz}
\usepackage{tikz-3dplot}
\usepackage{pifont}
\usepackage{aligned-overset}
\usepackage{longtable}
\usepackage{booktabs}
\usepackage{nicefrac}
\usepackage[usenames,dvipsnames]{pstricks}
\usetikzlibrary{arrows,shapes,automata,decorations.text,decorations.pathmorphing,backgrounds,positioning,fit}
\usepackage{enumitem}
\usepackage{color}
\usepackage{xcolor}
\usepackage{tabularx}
\usepackage{pgfplots}
\usepackage{pst-grad} 
\usepackage{pst-plot} 
\usetikzlibrary{positioning}
\usepackage{rotating}
\usepackage[colorlinks=true, allcolors=blue]{hyperref} 
\usepackage{hyperref}
\hypersetup{linkbordercolor=green}
\usepackage{dblfloatfix}
\usepackage{url,amsmath,setspace,amssymb}
\usepackage{tikz}
\usepackage{graphicx}
\usepackage{caption}
\usepackage{subcaption}
\usepackage{float}
\usepackage{soul}
\usepackage{epsfig}
\usepackage{longtable}
\usepackage{cite}

\allowdisplaybreaks
\IEEEoverridecommandlockouts
\usetikzlibrary{external}
\definecolor{royalblue}{RGB}{65, 105, 225}
\definecolor{seagreen}{RGB}{46, 139, 87}
\definecolor{firebrick}{RGB}{178,34,34}
\definecolor{darkviolet}{RGB}{138, 43, 226}
\definecolor{carrotorange}{RGB}{237, 145, 33}
\pgfplotsset{every tick label/.append style={font=\small}}

\title{\LARGE \bf Scalable Average Consensus with Compressed Communications}
\author{Mohammad Taha Toghani and C\'{e}sar A. Uribe
\thanks{The authors are with the Department of Electrical and Computer Engineering, Rice University, 6100 Main St, Houston, TX 77005, USA, \{mttoghani, cauribe\}@rice.edu.}
}

\begin{document}
\maketitle

\begin{abstract}
We propose a new decentralized average consensus algorithm with compressed communication that scales linearly with the network size $n$. We prove that the proposed method converges to the average of the initial values held locally by the agents of a network when agents are allowed to communicate with compressed messages. The proposed algorithm works for a broad class of compression operators (possibly biased), where agents interact over arbitrary static, undirected, and connected networks. We further present numerical experiments that confirm our theoretical results and illustrate the scalability and communication efficiency of our algorithm.
\end{abstract}

\section{Introduction}
\label{sec:introduction}

We consider the problem of decentralized average consensus over a network of $n$ agents, where each agent \mbox{$i\in[n]$} starting from an initial vector \mbox{$\vx_i\in\bbR^d$}, seeks to reach consensus on the global average through communication with its neighbors. Formally, the agents attempt to collaboratively solve the following optimization problem:
\vspace{-0.3em}
\begin{align}\label{eq:opt-prob}
    \vx^\star := \argmin_{\vx\in\bbR^d} \frac{1}{2n}\sum_{i=1}^{n}\lVert \vx-\vx_i\rVert^2,
\end{align}
only by sharing information with their local neighbors on the corresponding communication network.

The average consensus problem is at the core of many decentralized problems like inference~\cite{nedic2017fast} and optimization~\cite{Nedic2009DistributedSM,nedic2010constrained} which themselves are motivated by a wide range of applications such as decentralized federated learning~\cite{lalitha2018fully}, distributed localization and tracking~\cite{manley2006localization}, distributed sensor fusion~\cite{xiao2005scheme}, distributed time synchronization~\cite{syed2006time}, etc. These algorithms generally enjoy advantages like parallel computation, privacy, and resiliency to the central party's failure~\cite{kairouz2019advances}. However, they raise several important challenges such as the existence of adversaries, connection failure, synchronization, communication overhead, and scalability~\cite{wang2021field}.

In gossip type algorithms, each node \mbox{$i\in[n]$} builds a sequence \mbox{$\{\vx_i(t)\}_{t\geq 0}$} over the course of time, by interacting with its neighbors~\cite{xiao2004fast,cai2014average,Olshevsky2017LinearTA}. Given a set of initial parameters \mbox{$\vx_i(0)$}, for all \mbox{$i\in[n]$}, their objective is to solve~\eqref{eq:opt-prob}, i.e., reach consensus on \mbox{$\overline{\vx} := \big(\sum_{i=1}^{n}\vx_i{(0)}\big)/n$}. The convergence rate of such algorithms essentially depends on the connectivity of the network over which the agents interact~\cite{xiao2004fast}.

Decentralized consensus frameworks classically require agents to share their current estimates of the average value with their neighbors. This imposes a significant communication overhead on the network when $d$, the estimates' dimension, is large\cite{xiao2004fast,cai2014average}. To address this issue, several average consensus methods have been proposed under quantized communication techniques\cite{kashyap2007quantized,nedic2009distributed,frasca2009average, baldan2009efficient,thanou2012distributed,aysal2008distributed,carli2010quantized}, wherein the agents reduce the number of transmitted bits per communication round. However, convergence is generally not exact (i.e., only to some point close true average), or increasingly finer quantization is required. Recently, the authors in~\cite{cai2011quantized,koloskova2019decentralized} used an error-feedback scheme to provide algorithms with exact consensus to the average. Nevertheless, the dependency on the network topology and the number of agents is suboptimal. Recent studies have explored these phenomena in optimization and inference problems~\cite{kovalev2021linearly,taheri2020quantized,toghani2021communication,song2021compressed}.

The network size $n$ plays an essential role in the scalability of the gossip-type algorithms. Network structures with low connectivity, e.g., path and ring, have quadratic mixing times \mbox{$\mcO(n^2)$}\cite{nedic2019graph}, i.e., the number of iterations necessary for them to reach consensus grows quadratically with $n$. In~\cite{Olshevsky2017LinearTA}, the author suggested a momentum-based approach that implicitly improves the dependence of mixing time by a factor $n$. This technique has been extended to optimization and social learning problems~\cite{nedic2017fast,nedic2017achieving}.

In this paper, we jointly address the (i) \textit{communication-efficiency} and (ii) \textit{scalability} challenges for the decentralized average consensus problem. Motivated by~\cite{koloskova2019decentralized,Olshevsky2017LinearTA}, we propose a scalable algorithm that requires agents to communicate compressed messages using a class of randomized compression operators.  Prior efforts have proposed either scalable~\cite{Olshevsky2017LinearTA} or communication-efficient~\cite{koloskova2019decentralized} algorithms, while our work exploits both.

Our contributions can be summarized as follows:
\begin{itemize}[leftmargin=1em]
    \item We present a novel scalable and communication-efficient algorithm for the average consensus problem.
    \item Under an appropriate compression operator, we provide convergence guarantees for our proposed algorithm as well as an extension of the algorithm in~\cite{Olshevsky2017LinearTA}. Moreover, we show the convergence rate depends linearly on the number of nodes.
    \item We present the communication advantages of our algorithm through numerical results on two classes of networks with low connectivity.
\end{itemize}

The remainder of this paper is structured as follows. In Section~\ref{sec:setup}, describing the problem setup, we propose our algorithm, \textit{Scalable Compressed Gossip}, and state our theoretical results. In Section~\ref{sec:analysis}, we present the convergence proof for our algorithm. Section~\ref{sec:experiments} provides numerical results for the proposed algorithm. Finally, conclusions and future works are remarked in Section~\ref{sec:conclusion}.

\noindent\textbf{$\diamond$ Notation:} We write $[n]$ to denote the set \mbox{$\{1,\dots,n\}$}. We use the bolding notation for vectors and matrices. For a matrix \mbox{$\mX$}, we write $\mX_{ij}$ to denote the entry in the $i$-th row and $j$-th column. We use $\mI_n$ for the identity matrix of size $n\times n$ as well as $\vect{1}_n$ for the vector of all one with size $n$, where we may drop $n$ for brevity. We refer to agents by subscripts. We write $\lambda_i(\mW)$ to denote the \mbox{$i$-th} largest eigenvalue of matrix $\mW$, in terms of magnitude. We denote \mbox{$\lVert\vx\rVert$} and \mbox{$\lVert \mX \rVert_F$} respectively as $2$-norm of vector $\vx$ and frobenius norm of matrix $\mX$. We refer to matrix norm of a square matrix $\mW$ as \mbox{$\lVert \mW \rVert$}. We denote \mbox{$\mA \otimes \mB$} as the Kronecker product of any two matrices $\mA$ and $\mB$. We write $\vx(t)$ in reference to the value of parameter $\vx$ at time $t$. 

\section{Problem Setup, Algorithm, \& Results}\label{sec:setup}

This section first states the communication setup and describes the class of compression operators used by the proposed algorithm. We then present our scalable and communication-efficient algorithm and provide its convergence analysis.

\noindent\textbf{$\diamond$ Communication Network:} Consider a set of $n$ agents interacting over a fixed, undirected, and connected communication network \mbox{$\mcG = \{[n],\mcE\}$}, where \mbox{$\mcE\subseteq [n]\times[n]$} is the set of edges. If there is a link between any two agents $i$ and $j$, then they may exchange information with each other. We denote $\mcN_i$ as the set of agent $i$'s neighbors as well as $\mcN_i' = \mcN_i\cup\{i\}$, for all $i\in[n]$. We denote matrix \mbox{$\mW\in [0,1]^{n\times n}$} with positive diagonal entries, a proper \textit{mixing matrix} corresponding to network $\mcG$, if it is symmetric (\mbox{$\mW=\mW^\top$}), doubly stochastic (\mbox{$\mW\vect{1}=\mW^\top\vect{1}=\vect{1}$}), and \mbox{$\mW_{ij}=0$} for \mbox{$(i,j)\notin\mcE$}, \mbox{$i\neq j$}. We also denote \mbox{$\delta(\mW)$} as the \textit{spectral gap} of matrix $\mW$, i.e., the gap between the first and second largest eigenvalues of $\mW$, which lies in \mbox{$(0,1]$}. Furthermore, given an undirected graph \mbox{$\mcG = \{\mcV,\mcE\}$}, we define its associated \textit{Metropolis–Hasting} mixing matrix \mbox{$\mW = \mathcal{MH}(\mcG)$}~\cite{nonaka2010hitting} as follows:
\begin{align}\label{eq:lazy-metropolis-matrix}
\mW_{ij} =
\begin{cases}
{1}/{\max\big\{|\mcN_i'|,|\mcN_j'|\big\}}, & \text{if}~(i,j)\in\mcE\vspace{0.3em}\\
1-\displaystyle\sum_{j\neq i}\mW_{ij}, &  \text{if}~i=j\vspace{-0.5em}\\
0,& \text{otherwise}
\end{cases}
\end{align}

\noindent\textbf{$\diamond$ Compression Operator:}
Here, we introduce a class of compression operators that has been widely studied for distributed optimization~\cite{alistarh2017qsgd,koloskova2019decentralized,beznosikov2020biased}. We assume the compression operator \mbox{$Q:\bbR^m \times \mcZ \times [0,1) \rightarrow \bbR^m$} satisfies
\begin{align}\label{eq:q-comp}
\bbE_{\vzeta}\big\lVert Q(\vx,\vzeta,\omega) - \vx \big\rVert^2 \leq \omega^2 \big\lVert \vx \big\rVert^2, \qquad \forall \vx\in \bbR^d,
\end{align}
where \mbox{$\omega\in[0,1)$}, \mbox{$\vzeta$} is a random variable with output space $\mcZ$, and \mbox{$\bbE_{\vzeta}[.]$} indicates the expectation over the internal randomness of $Q$. Note that in~\eqref{eq:q-comp}, \mbox{$\omega=0$} implies no compression (i.e., exact communications). Hereafter, we drop $\vzeta,\omega$ from $Q$ and $\bbE$ for simplicity of notation.

The class of randomized operators introduced in~\eqref{eq:q-comp} embraces a wide range of functions, both sparsification, and quantization, some of which we mention in Example~\ref{exam:q-comp}\cite{toghani2021communication}.

\begin{example}\label{exam:q-comp} The following operators fulfill~\eqref{eq:q-comp}:

\begin{itemize}[leftmargin=1em]
\item \mbox{$\mathrm{rand}_{k}$}: Select $k$ out of $d$ coordinates randomly and mask the rest to zero, \mbox{$\omega^2=1{-}k/d$}.
\item \mbox{$\mathrm{top}_{k}$}: Select $k$ out of $d$ coordinates with highest magnitude and mask the rest to zero, \mbox{$\omega^2=1{-}k/d$}.
\item \mbox{$\mathrm{qsgd}_{k}$}: Round each coordinate of \mbox{${\lvert\vx\rvert}/{\lVert\vx\rVert}$} to one of the \mbox{$u=2^{k{-}1}{-}1$} quantization levels (\mbox{$k{-}1$} bits), and one bit for the sign of the coordinate, i.e.,
\begin{align*}
    \mathrm{qsgd}_{k}(\vx) &= \frac{\sign(\vx).\lVert\vx\rVert}{u\tau}\left\lfloor u\frac{\lvert\vx\rvert}{\lVert\vx\rVert}+\vzeta\right\rfloor, \qquad \vzeta{\sim}[0,1]^d,
\end{align*}
where \mbox{$\tau=1{+}\min\big\{{d}/{u^2},{\sqrt{d}}/{u}\big\}$}, and \mbox{$\omega^2 = 1{-}\tau^{-1}$}.
\end{itemize}
\end{example}

We next propose our method and discuss its features.

\begin{algorithm}[!t]
    \caption{Scalable Compressed Gossip (\textit{SCG})}\label{alg:scg}
	{\textbf{input:} initial parameters \mbox{$\vx_i(0)\in\bbR^d$}, \mbox{for all $i\in[n]$}, network \mbox{$\mcG=(\mcV,\mcE)$} with mixing matrix $\mW$, stepsize \mbox{$\gamma \in (0,1]$}, operator $Q$ with \mbox{$\omega\in [0,1)$}, momentum \mbox{$\sigma\in [0,1)$}.}\\
	
	\vspace{-1.2em}
	\begin{algorithmic}[1]
	    \STATE{$\hat{\vx}_i(0) := \vect{0}$, $\vy_i(0):=\vx_i(0), \quad\forall i\in[n]$}
	    \FOR{$t$ \textbf{in} $0,\dots,t{-}1$, in parallel $\forall i \in [n]$}
	    \STATE{$\vq_i(t):= Q(\vx_i(t) - \hat{\vx}_i(t))$}\label{ln:diff-comp}
	    \STATE{Send $\vq_i(t)$ and receive $\vq_j(t)$, \quad for all $j\in\mcN_i$}\label{ln:communicate}
	    \STATE{$\hat{\vx}_j(t{+}1):=\hat{\vx}_j(t)+\vq_j(t)$, \quad for all $j\in\mcN_i'$}\label{ln:update-x-hat}
	    \STATE{$\vy_i(t{+}1) := \vx_i(t) + \gamma\sum\limits_{j\in\mcN_i'}  \mW_{ij}\left(\hat{\vx}_j(t{+}1) - \hat{\vx}_i(t{+}1)\right)$}\label{ln:update-y}
	    \STATE{\vspace{-0.4em}$\vx_i(t{+}1) := (1{+}\sigma)\,\vy_i(t{+}1) - \sigma\,\vy_i(t)$}\label{ln:update-x}
	   \ENDFOR
	\end{algorithmic}
\end{algorithm}

\noindent\textbf{$\diamond$ Algorithm:} We here present our communication-efficient and scalable gossip type algorithm. As we discussed in Section~\ref{sec:introduction}, let \mbox{$\vx_i(t)$} be the vector belonging to agent $i$ at time $t$, for all \mbox{$i\in[n]$} and \mbox{$t\geq 0$}. Similar to~\cite{koloskova2019decentralized},  we consider an error-feedback framework, wherein each agent $i$ gradually estimates \mbox{$\hat\vx_j(t)$}, an approximation of its neighbors' parameters \mbox{$\vx_j(t)$} (including itself), for all \mbox{$j\in\mcN_i'$}. Algorithm~\ref{alg:scg} presents a detailed pseudo-code for our method. Each agent $i$ begins with an initial $\vx_i(0)$ and a slack parameter \mbox{$\vy_i(0)=\vx_i(0)$}, besides $\hat\vx_j(0)=\vect{0}$. Lines~\ref{ln:diff-comp}-\ref{ln:update-x} of Algorithm~\ref{alg:scg} describe the operations for each round of the algorithm. In a nutshell, agent $i$ at round $t$, (i) computes a compressed version \mbox{$\vq_i(t)$} of the difference between \mbox{$\vx_i(t)$} and \mbox{$\hat\vx_i(t)$}, (ii) exchanges compressed vectors \mbox{$\vq_i(t)$} and \mbox{$\vq_j(t)$} with each neighbor \mbox{$j\in\mcN_i$}, (iii) uses \mbox{$\vq_j(t)$} to update \mbox{$\hat\vx_j(t{+}1)$}, for all \mbox{$j\in\mcN_i'$}, then (iv) updates \mbox{$\vy_i(t{+}1)$} based on \mbox{$\vx_i(t)$} and \mbox{$\hat\vx_j(t{+}1)$}, for all \mbox{$j\in\mcN_i'$}, and finally (v) extrapolates \mbox{$\vx_i(t{+}1)$} based on \mbox{$\vy_i(t{+}1)$} and \mbox{$\vy_i(t)$}. 

We now state a matrix notation for our algorithm. Let \mbox{$\mX(t)=\big[\vx_1(t),\dots,\vx_n(t)\big]^\top$}, \mbox{$\hat{\mX}(t)=\big[\hat{\vx}_1(t),\dots,\hat{\vx}_n(t)\big]^\top$}, \mbox{$Q(\mX)=\big[Q(\vx_1),\dots,Q(\vx_n)\big]^\top$}, \mbox{$\overline{\mX}=\big[\overline{\vx},\dots,\overline{\vx}\big]^\top$}, as well as \mbox{$\mY(t)=\big[\vy_1(t),\dots,\vy_n(t)\big]^\top$}, all be matrices of size \mbox{$n\times d$}. Then, Algorithm~\ref{alg:scg} may be written as follows:
\begin{align}\label{eq:update-scg}
\begin{split}
    \hat{\mX}(t{+}1) &\coloneqq \hat{\mX}(t) + Q\big(\mX(t)-\hat{\mX}(t)\big),\\
    \mY(t{+}1) &\coloneqq \mX(t) + \gamma\big(\mW{-}\mI\big)\hat{\mX}(t{+}1),\\
    \mX(t{+}1) &\coloneqq (1{+}\sigma) \mY(t{+}1) - \sigma \mY(t),
\end{split}
\end{align}
with \mbox{$\mY(0)=\mX(0)$}. Given the fact that matrix $\mW$ is doubly stochastic, we can see that \mbox{$\frac{\vect{1}\vect{1}^\top}{n}\mX(t) = \frac{\vect{1}\vect{1}^\top}{n}\mY(t) = \overline{\mX}$}, for all $t\geq 0$. In other words, Algorithm~\ref{alg:scg} maintains the mean of $\mX(t)$ and $\mY(t)$ constant.
\vspace{0.5em}

\noindent\textbf{$\diamond$ Comparison:} Algorithm~\ref{alg:scg} implicitly yields the following three methods\footnote{\textit{SEG} with the update rule in~\eqref{eq:update-seg}, is an extension of the algorithm in~\cite{Olshevsky2017LinearTA} which we analyze its convergence in Theorem~\ref{thm:seg}.}:
\begin{itemize}[leftmargin=1em]
    \item \small\textbf{Exact Gossip (EG)~\cite{xiao2004fast}:} $\sigma=0$, $\omega=0$,\vspace{0.3em}
    \item \small\textbf{Compressed Gossip (CG)~\cite{koloskova2019decentralized}:} $\sigma=0$, $\omega\in[0,1)$,\vspace{0.3em}
    \item \small\textbf{Scalable Exact Gossip (SEG)~\cite{Olshevsky2017LinearTA}:} $\sigma=\frac{5n-\sqrt{\gamma}}{5n+\sqrt{\gamma}}$, $\omega=0$.\vspace{0.3em}
\end{itemize}

Note that \textit{SEG} and \textit{SCG} require the agents know the network size $n$ or some \mbox{$U=\mcO(n)$} to compute $\sigma$ (see~\cite{Olshevsky2017LinearTA}).

Before stating the main results, let us compare our algorithm with prior works. Table~\ref{tab:gossip-comparison} illustrates the linear convergence rates of the algorithms mentioned above along with a conservative bound for their feasible step-size $\gamma$ and compression ratio $\omega$. First, \textit{EG} and \textit{SEG} linear rates, which require exact communication, have a quadratic and linear dependence on $n$ respectively. We will discuss in Section~\ref{sec:analysis} how $\gamma$ impacts the spectral gap of the mixing matrix in~\eqref{eq:update-scg}. Second, \textit{CG} enjoys an arbitrary compression with a rate of $\mcO(n^4)$, but the choice of $\gamma$ is limited to $\mcO(n^{-4})$. In this work, we use a different technique to analyze our algorithm \textit{SCG}, where we restrict the choice of $\omega$ and let $\gamma$ be arbitrary. As shown in Table~\ref{tab:gossip-comparison}, \textit{CG} and \textit{SCG} enjoy the same convergence rates as \textit{EG} and \textit{SEG}, with bounded $\gamma$. Given a reasonable bound for $\omega$, our algorithm has a better dependence on $n$ than \textit{CG} given the same step-size $\gamma$. We conjecture that $\gamma$ offers a trade-off between the convergence rate and the value of $\omega$. In other words, with decreasing $\gamma$ proportional to $n^{{-}1}$, the feasible set for $\omega$ expands proportional to $n$, which implies a worse convergence rate dependence on $n$. We will discuss this trade-off in Fig.~\ref{fig:gossip-qsgd-feasibility}.

\begin{table}[!t]
\caption{Comparison of the worst case convergence rates for \textit{EG}, \textit{SEG}, \textit{CG}, and \textit{SCG}.}
\label{tab:gossip-comparison}
\vspace{-0.7em}
\begin{center}
\begin{minipage}{\linewidth}
\centering
\resizebox{\linewidth}{!}{
\begin{tabular}{clcc} \toprule
Algorithm & Linear Rate\footnote{Convergence rates are linear, with different dependence on $n$ and $\gamma$. Rates are presented for the worst case graphs where \mbox{$\delta(\mW)=\mcO(n^{-2})$}.} & Stepsize ($\gamma$) & $\omega$\\

\midrule
\midrule
\begin{tabular}{c} \textbf{EG}~\cite{xiao2004fast} \end{tabular}& 
$\mcO\big(1{-}\gamma n^{-2}\big)$ & $(0,1]$ & $0$ \\[2pt]

\midrule
\begin{tabular}{c} \textbf{SEG}~\cite{Olshevsky2017LinearTA} \end{tabular}& $\mcO\big(1{-}\gamma^{\frac{1}{2}}n^{-1}\big)$ & $\big(0,\frac{1}{2}\big]$ & $0$\\[2pt]

\midrule
\begin{tabular}{c} \textbf{CG}~\cite{koloskova2019decentralized}\end{tabular}& 
$\mcO\big(1{-}n^{-4}\big)$& $\mcO\big(n^{-4}\big)$& $[0,1)$\\[2pt]

\midrule
\begin{tabular}{c} \textbf{CG}\footnote{An alternative analysis for \textit{CG} with bounded $\omega$ and flexible $\gamma$.}\end{tabular}& 
$\mcO\big(1{-}\gamma n^{-2}\big)$ & $(0,1]$ &$\Big[0,\Theta\Big(\frac{1}{(1{+}\gamma)n^{2}}\Big)\Big]\,\,$\\[2pt]


\midrule
\begin{tabular}{c} \textbf{SCG}\\
{\small\textbf{\color{magenta}This Work}} \end{tabular}& $\mcO\big(1{-}\gamma^{\frac{1}{2}}n^{-1}\big)$ & $\big(0,\frac{1}{2}\big]$ & $\Big[0,\Theta\Big(\frac{1}{(1{+}\gamma)n^{2}}\Big)\Big]$\footnote{Asymptotic bound for $\omega$ in Theorem~\ref{thm:scg}.}\\[2pt]
\bottomrule[1pt]
\end{tabular}
}
\end{minipage}
\end{center}

\vspace{-2em}
\end{table}

\noindent\textbf{$\diamond$ Main Results:} Here, we first propose the convergence guarantees for \textit{SEG} and then \textit{SCG}. As we mentioned earlier, under \mbox{$\omega=0$},~\eqref{eq:update-scg} turns into the update rule for \textit{SEG}:
\begin{align}\label{eq:update-seg}
\begin{split}
    \mY(t{+}1) &\coloneqq \mX(t) + \gamma\big(\mW{-}\mI\big)\mX(t),\\
    \mX(t{+}1) &\coloneqq (1{+}\sigma) \mY(t{+}1) - \sigma \mY(t),
\end{split}
\end{align}
The following theorem states the convergence rate for~\eqref{eq:update-seg}.

\begin{theorem}[An extension of Theorem~2.1 from~\cite{Olshevsky2017LinearTA}]\label{thm:seg}
Let stepsize $\gamma\in\big(0,\frac{1}{2}\big]$, \mbox{$\mY{(0)} = \mX{(0)}$}, and \mbox{$\mW = \mathcal{MH}(\mcG)$}. The following property holds for the update rule in~\eqref{eq:update-seg}:
\begin{align*}
    \Psi_{x}(t) \leq 2\lambda^t \Psi_{x}{(0)},
\end{align*}
where \mbox{$\Psi_{x}(t){=}\big\lVert \mX(t) {-} \overline{\mX}\big\rVert_F$}, \mbox{$\lambda{=}1{-}\frac{\sqrt{\gamma}}{5n}$}, when \mbox{$\sigma{=}\frac{5n-\sqrt{\gamma}}{5n+\sqrt{\gamma}}$}.
\end{theorem}

The result in Theorem~\ref{thm:seg} holds for an arbitrary stepsize $\gamma\in(0,{1}/{2}]$, compared to~\cite{Olshevsky2017LinearTA} that holds for $\gamma=1/2$ only. The auxiliary mixing matrix used by both \textit{SEG} and \textit{SCG} have the same dependence on $\gamma$, so the analysis for \textit{SEG} helps to understand the analysis for \textit{SCG} better.

\begin{theorem}[SCG Convergence Analysis]\label{thm:scg} Let compression operator $Q$ satisfies~\eqref{eq:q-comp}, \mbox{$\mY{(0)} = \mX{(0)}$}, \mbox{$\hat{\mX}{(0)}= \vect{0}$}, and $\gamma$, $\sigma$, $\lambda$, and $\mW$ be as Theorem~\ref{thm:seg}. Then, the update rule in~\eqref{eq:update-scg} satisfies the following: for \mbox{$\omega \leq \big(2\big(\kappa_3+\gamma\beta\kappa_2\big)\big(\lambda^{-\frac{1}{2}}+\gamma \beta\kappa_2C\lambda^{-1}(1-\lambda^\frac{1}{2})^{-2}\big)\big)^{-1}$}
\begin{align*}
    \bbE\Psi_{x}(t) \leq C_0\tilde{\lambda}^t \Psi_{x}{(0)},
\end{align*}
where \mbox{$\kappa_2 {=} \sqrt{2\sigma^2{+}2\sigma{+}1}$}, \mbox{$\kappa_3=\sqrt{2\sigma^2{+}2}$}, \mbox{$\beta=\lVert\mW{-}\mI\rVert$}, \mbox{$\tilde{\lambda}{=}1{-}\frac{\sqrt{\gamma}}{10n}$}, \mbox{$\Psi_{x}(t){=}\big\lVert \mX(t) {-} \overline{\mX}\big\rVert_F$}, and constants $C_0,C>0$.
\end{theorem}

The above theorem implies a linear convergence for Algorithm~\ref{alg:scg} with rate $\tilde{\lambda}$ dependent on \mbox{$\gamma^{-{1}/{2}} n$} under a bounded compression ratio $\omega$, where the bound on $\omega$ can be written as $\Theta\big((1{+}\gamma)^{-1}n^{-2}\big)$.
The above bound suggests that the consensus step-size $\gamma$ imposes a trade-off between the convergence rate and the compression ratio $\omega$. The proofs for both theorems are presented in Section~\ref{sec:analysis}.


\section{Convergence Analysis}
\label{sec:analysis}
Before stating the proofs, we propose a technical lemma that will help us prove Theorems~\ref{thm:seg} and~\ref{thm:scg}.
\begin{lemma}\label{lem:mixing-matrix-2n}
Let matrix {$\mA\in\bbR^{n\times n}$} be symmetric, doubly stochastic, and diagonally dominant with \mbox{$\lambda_2(\mA)\leq 1{-}\frac{1}{p^2}$}, for some $p>1$. 
and \mbox{$\mB\in \bbR^{2n\times 2n}$} be as follows:
\begin{align*}
    \mB = \begin{bmatrix} (1{+}\sigma) \mA & -\sigma\mA\\
    \mI & \vect{0}\end{bmatrix}.
\end{align*}
Let \mbox{$\lambda= 1{-}\frac{1}{p}$} and \mbox{$\sigma=\frac{p-1}{p+1}$}, then following statements hold:
\begin{enumerate}[label=(\alph*)]
    \item\label{lem:mixing-equal}\cite[Lemma~2.5]{Olshevsky2017LinearTA} If \mbox{$\vv = [\vq^\top,\vq^\top]^\top$} and \mbox{$\overline{\vv} = [\overline{\vq}^\top,\overline{\vq}^\top]^\top$}, for an arbitrary \mbox{$\vq\in\bbR^n$} with \mbox{$\overline{\vq} = \frac{\vect{1}\vect{1}^\top}{n}\vq$},  then \mbox{$t\geq 0$},
    \begin{align*}
        \lVert \mB^t \vv - \overline{\vv}\rVert \leq 2\lambda^t\lVert \vv - \overline{\vv}\rVert.
    \end{align*}
    \item\label{lem:mixing-zero-mean} If \mbox{$\vv = [\vq^\top,\vect{0}^\top]^\top$}, for $\vq\in\bbR^n$ such that $\vect{1}^\top\vq=0$, then for all \mbox{$t\geq 0$},
    \begin{align*}
        \lVert \mB^t \vv \rVert \leq C t\lambda^t,
    \end{align*}
    where $C>0$ is some constant.
\end{enumerate}
\end{lemma}
\begin{proof}[Proof sketch for Lemma~\ref{lem:mixing-matrix-2n}]
Similar to\cite[Lemma~2.3]{Olshevsky2017LinearTA}, by considering the SVD-decomposition of $\mA$, the problem reduces to show the convergence of \mbox{$[\mB(\lambda)]^t\vr$}, for \mbox{$\vr = [1,1]^\top$}, and \mbox{$\vr = [1,0]^\top$}, where \mbox{$\mB(\lambda_i)=$\resizebox{0.24\linewidth}{!}{$\begin{bmatrix} (1{+}\sigma) \lambda_i & -\sigma\lambda_i\\ 1 & 0\end{bmatrix}$}} is a \mbox{$2\times 2$} matrix, for $i\in\{2,3,\dots,n\}$, and \mbox{$1-\lambda_i\geq {p^{-2}}$}.
\vspace{0.1em}

The analysis in~\cite{Olshevsky2017LinearTA} shows the convergence for \mbox{$\vr = [1,1]^\top$}, but their method is restricted to vectors $\vr$ with the same two elements. However, this is not the case for Lemma~\ref{lem:mixing-matrix-2n}\ref{lem:mixing-zero-mean}, thus we consider an alternative technique. Note that \mbox{$[\mB(\lambda)]^t\vr$} implies a recursive sequence with the following definition:
\begin{align}\label{eq:recursive-seq}
    a(t) = (1{+}\sigma)\lambda\, a({t{-}1}) - \sigma\lambda\, a({t{-}2}), \quad \text{for all } t\geq 2,
\end{align}
with $a(1)=1$, and either \mbox{$a(0)=0~\text{or}~1$}. To find $a(t)$, we consider its corresponding generating function
\begin{align}\label{eq:generating-func}
    G(x) = \frac{[a(1)-(1{+}\sigma)\lambda\,a(0)]\,x+a(0)}{\sigma\lambda\,x^2 - (1{+}\sigma)\lambda\,x + 1},
\end{align}
where one can find the exact form of $a(t)$ given the choices for $a(1)$ and $a(0)$. The exact solution for $a(t)$ completes the proof for Lemma~\ref{lem:mixing-matrix-2n}.
\end{proof}

We need Lemma~\ref{lem:mixing-matrix-2n} in the proof for both theorems, and Lemma~\ref{lem:mixing-matrix-2n}\ref{lem:mixing-zero-mean} for Theorem~\ref{thm:scg}.
Next, we show the proof sketch for Theorem~\ref{thm:seg}.

\begin{proof}[Proof sketch for Theorem~\ref{thm:seg}]
Let \mbox{$\mM=(1{-}\gamma)\mI+\gamma\mW$}, be a lazy version of $\mW$ defined in~\eqref{eq:lazy-metropolis-matrix}, thus $\mM$ is also a doubly stochastic matrix with \mbox{$\delta(\mM)=\gamma\delta(\mW)$}. We seek to derive a lower bound of \mbox{$\mcO\big({\gamma}/{n^2}\big)$} on the spectral gap of matrix $\mM$. 
Our proof follows the structure of~\mbox{\cite[Theorem~2.1]{Olshevsky2017LinearTA}}, but we consider an arbitrary $\gamma\in(0,1/2]$, which will also be used for Theorem~\ref{thm:scg}. Note that a doubly stochastic matrix can be interpreted as a Markov chain's transition matrix. Now, assume that $\mM$ is the transition matrix associated with a Markov chain. We know that $\mM$ is a convex combination of $\mI$ and $\mW$, which implies with probability $\gamma$, the matrix $\mW$ determines the transitions of the chain. Hence, using the result in~\cite{nonaka2010hitting}, we can infer that the following property holds for the hitting time\footnote{For a Markov Chain with transition matrix $\mW$, hitting time \mbox{$\mcH_{\mW}\left(i\rightarrow j\right)$} indicates the expected number of steps for the chain to reach state $j$ starting from state $i$.} of $\mM$~\cite{levin2017markov}:
\begin{align}~\label{eq:hitting-time-M}
    \max_{i,j\in[n]} \mcH_{\mM}(i \rightarrow j) \leq \frac{6n^2}{\gamma}.
\end{align}

Moreover, by~\mbox{\cite[Theorem~12.4 and Theorem~10.14]{levin2017markov}},
\begin{align}\label{eq:hitting-time-inequalities}
    \left(\frac{1}{\delta(\mM)}{-}1\right)\ln{2}&\leq 2\max_{i,j\in[n]} \mcH_{\mM}(i \rightarrow j)+1,
\end{align}
so, due to~\eqref{eq:hitting-time-M} and~\eqref{eq:hitting-time-inequalities}, we have \mbox{$\delta(\mM)\geq {\gamma}/{25n^2}$}. The rest of the proof is an immediate result of Lemma~\ref{lem:mixing-matrix-2n}\ref{lem:mixing-equal}.
\end{proof}

\begin{figure*}[!ht]
    \centering
    \begin{minipage}{0.33\textwidth}
    \includegraphics[width=1\linewidth]{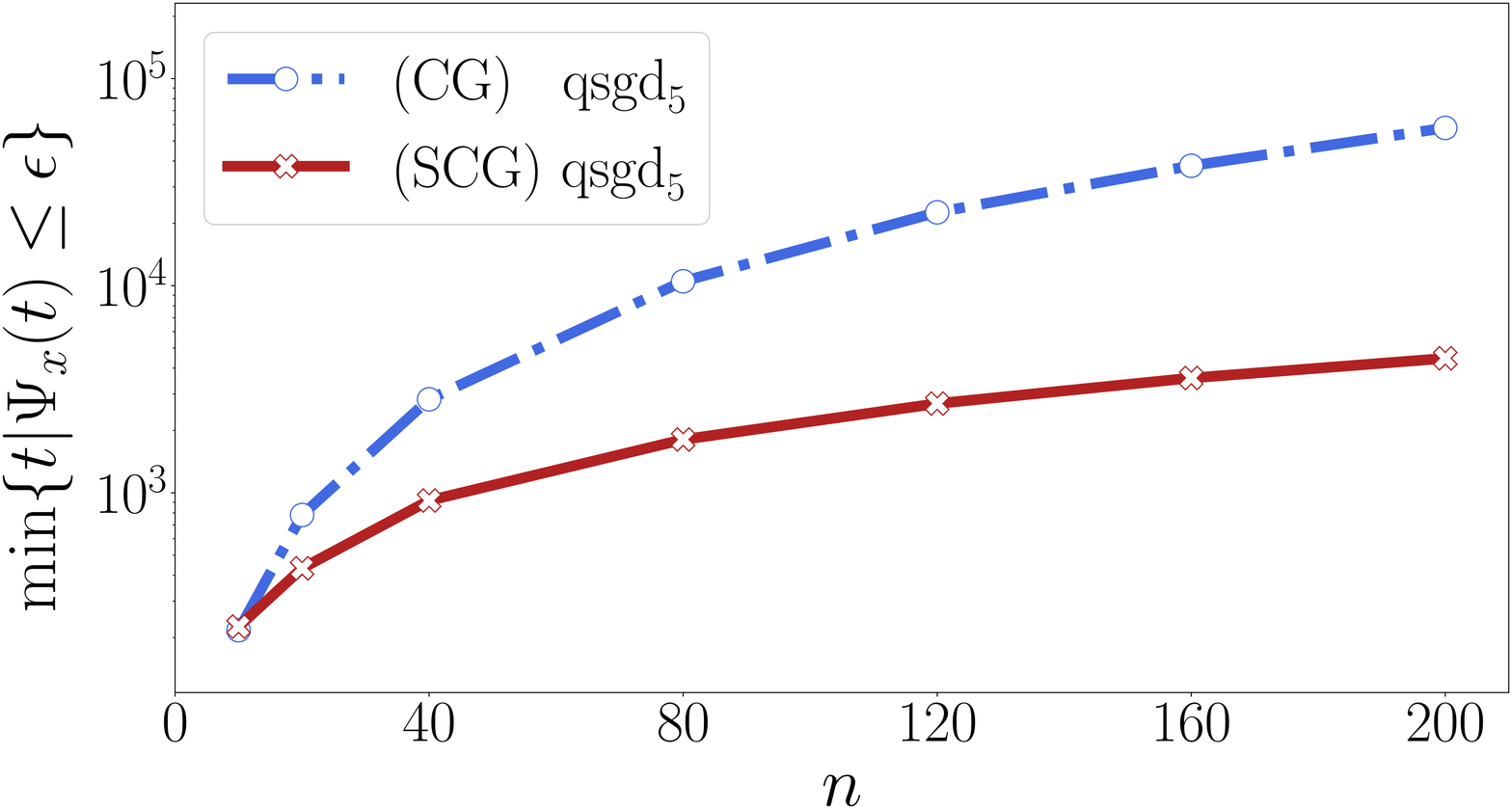}
    \subcaption{Path graphs with size $n$ from $10$ to $200$}
    \label{fig:gossip-qsgd-n-iteration}
    \end{minipage}
    \hspace{0.2em}
    \vline
    \hspace{0.2em}
    \begin{minipage}{0.62\textwidth}
    \includegraphics[width=0.5295\linewidth]{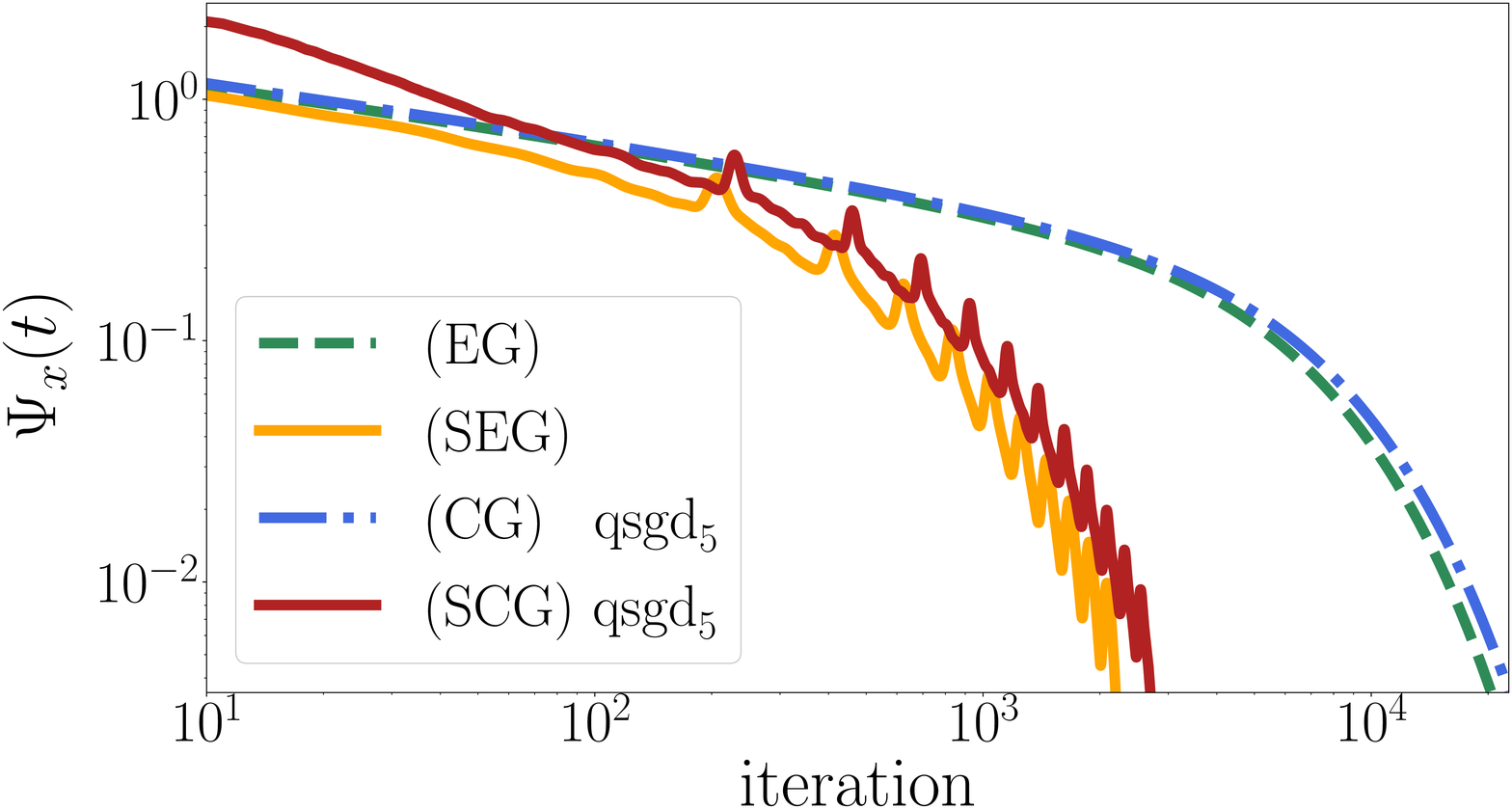}
    \includegraphics[width=0.5\linewidth]{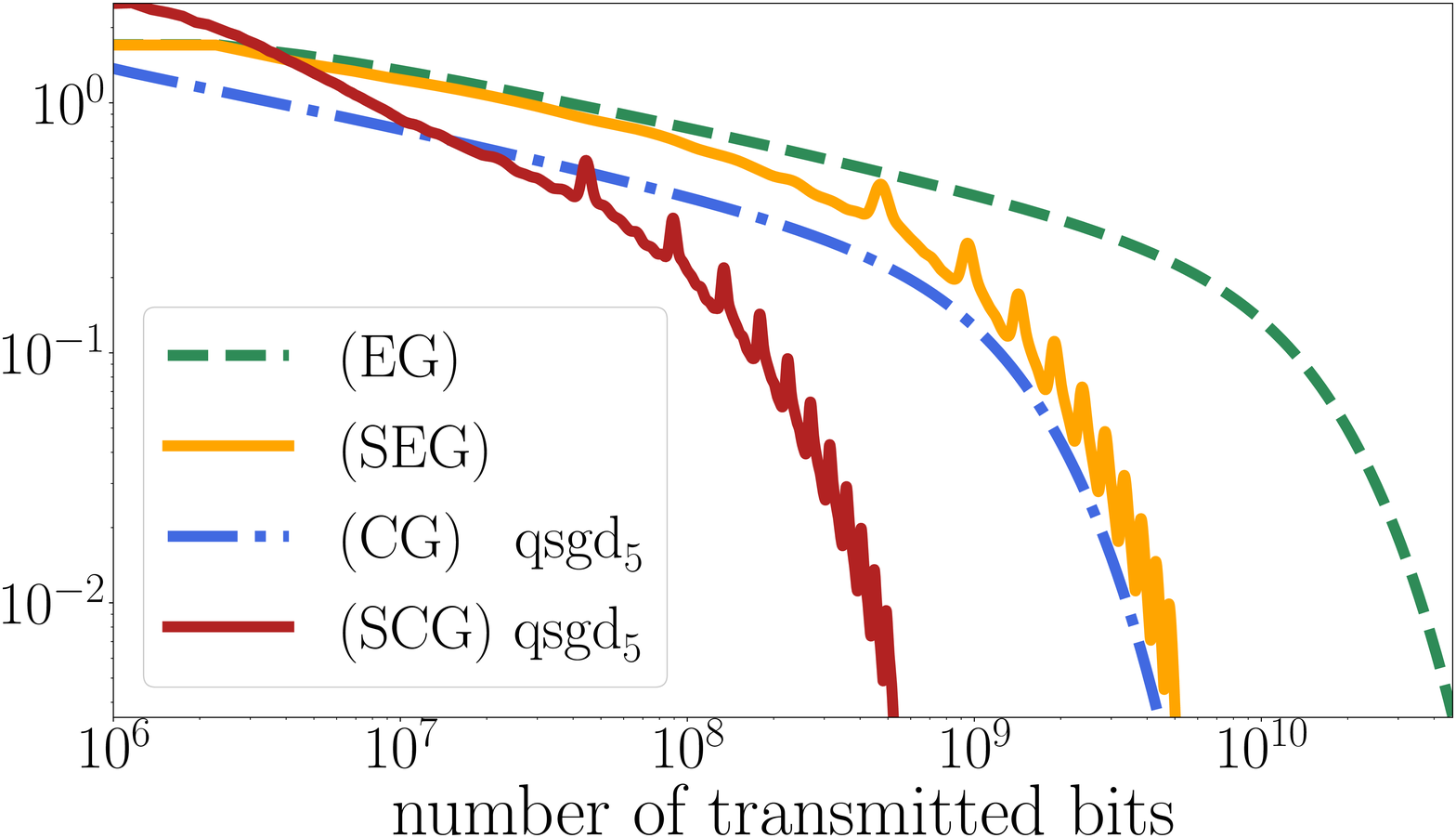}
    \subcaption{Ring network with size $n=120$}
    \label{fig:gossip-qsgd-iteration-bit}
    \end{minipage}
    \caption{\textbf{Scalability Numerical Analysis:} Each experiment is the average of 10 runs. \textbf{(a)} Comparison between the number of iteration required for \textit{CG} and \textit{SCG} to reach an \mbox{$\epsilon$-convergence} on the average consensus problems with \mbox{$d=150$}, for path networks with $n$ ranging from $10$ to $200$,  \mbox{$\mathrm{qsgd}_{5}$}, and $\epsilon=10^{-4}$. \textbf{(b)} Comparison of the $\epsilon$-suboptimality for algorithms in Table~\ref{tab:gossip-comparison}, for an average consensus problem with $d=150$, $\mathrm{qsgd}_{5}$, $\epsilon=10^{-4}$, over a ring graph with size $n=120$ based on the number of iterations (left) and the number of transmitted bits (right).}
    \label{fig:gossip-qsgd}
\end{figure*}

We now present the proof for Theorem~\ref{thm:scg}.

\begin{proof}[Proof for Theorem~\ref{thm:scg}]
Let us define $2\times 2$ matrices $\mT_1$, $\mT_2$, and $\mT_3$ as follows:
\begin{align}\label{eq:two-by-two-matrices}
    \mT_1 &= \begin{bmatrix}
    1{+}\sigma & -\sigma\\
    1 & 0
    \end{bmatrix}, \quad &&\\
    \mT_2 &= \begin{bmatrix}
    1{+}\sigma & -\sigma\\
    0 & 0
    \end{bmatrix}, \quad &&\kappa_2 = \lVert\mT_2\rVert = \sqrt{2\sigma^2{+}2\sigma{+}1},\nonumber\\
    \mT_3 &= \begin{bmatrix}
    \sigma & -\sigma\\
    1 & -1
    \end{bmatrix}, \quad &&\kappa_3=\lVert\mT_3\rVert = \sqrt{2\sigma^2{+}2}.\nonumber
\end{align}
We furthermore define $\mZ(t)$, $\hat{\mZ}(t)$, and $\overline{\mZ}$, as
\begin{align}\label{eq:z-definition}
\begin{split}
    \mZ(t) = 
    \begin{bmatrix}
    \mX(t)\\
    \mX{(t{-}1)}
    \end{bmatrix}, \,\,\,
    \hat{\mZ}(t) = 
    \begin{bmatrix}
    \hat{\mX}(t)\\
    \hat{\mX}{(t{-}1)}
    \end{bmatrix}, \,\,\,
    \overline{\mZ} = 
    \begin{bmatrix}
    \vspace{-0.2em}\overline{\mX}\\
    \overline{\mX}
    \end{bmatrix},
\end{split}
\end{align}
with initialization \mbox{$\hat{\mX}{(0)}= \vect{0}$} and \mbox{$\mX{(-1)} = \mX{(0)}$}. Therefore, the update rule in~\eqref{eq:update-scg} can be rewritten as follows:
\begin{align}\label{eq:z-update}
    &\mZ(t{+}1)\\
    &= [\mT_1\otimes\mI] \mZ(t) + \gamma[\mT_2\otimes\left(\mW{-}\mI\right)]\hat{\mZ}(t{+}1)\nonumber\\
    &= \mB\mZ(t) + \gamma[\mT_2\otimes\left(\mW{-}\mI\right)]\big(\hat{\mZ}(t{+}1){-}\mZ(t)\big)\nonumber\\
    &= \mB^{t{+1}}\mZ{(0)} {+} \gamma\sum_{s=0}^{t}\mB^s[\mT_2{\otimes}(\mW{-}\mI)](\hat{\mZ}{(t{-}s{+}1)}{-}\mZ{(t{-}s)}).\nonumber
\end{align}

We now define the Lyapunov functions $\mcR_{z}(t)$ and $\mcU_{z}(t)$ as
\begin{align*}
  \mcR_{z}(t) \triangleq \big\lVert \mZ(t) - \overline{\mZ}\big\rVert_F, \quad \mcU_{z}(t) \triangleq \big\lVert \hat{\mZ}(t{+}1) - \mZ(t) \big\rVert_F,
\end{align*}
and bound them using Lemma~\ref{lem:mixing-matrix-2n}. First, we have
\begin{align}\label{eq:acg-proof-1}
    &\big\lVert \mZ(t{+}1)-\overline{\mZ}\big\rVert_F
    \overset{\text{\tiny tri. ineq.}}{=}\big\lVert \mB^{t{+}1}\mZ{(0)}{-}\overline{\mZ}\big\rVert\nonumber\\%
    &+\gamma\sum_{s=0}^{t}\big\lVert\mB^s[\mT_2\otimes(\mW{-}\mI)]\big(\hat{\mZ}{(t{-}s{+}1)}-\mZ{(t{-}s)}\big) \big\rVert_F \nonumber\\
    \overset{\text{\tiny Lemma\,\ref{lem:mixing-matrix-2n}\ref{lem:mixing-equal},\,\ref{lem:mixing-zero-mean}}}&{\leq} 2\lambda^{t{+}1} \big\lVert \mZ{(0)} \big\rVert_F \nonumber\\
    &+\gamma \beta\kappa_2 C\sum_{s=0}^{t} s\lambda^s \big\lVert\hat{\mZ}{(t{-}s{+}1)}-\mZ{(t{-}s)}\big\rVert_F.
\end{align}

Using the definition of $\hat{\mX}(t)$ in~\eqref{eq:update-scg}, we also have
\begin{align}\label{eq:acg-proof-2-1}
    &\bbE \big\lVert \mZ(t{+}1)-\hat{\mZ}{(t+2)}\big\rVert_F^2\nonumber\\ \overset{\eqref{eq:z-definition}}&{=} \bbE \big\lVert \mX(t{+}1) - \hat{\mX}(t{+}2) \big\rVert_F^2 + \bbE \big\lVert \mX(t) - \hat{\mX}(t{+}1)\big\rVert_F^2\nonumber\\
    &\overset{\eqref{eq:q-comp}}{\leq} \omega^2\left[\big\lVert \mX(t{+}1)-\hat{\mX}(t{+}1)\big\rVert_F^2 + \big\lVert \mX(t)-\hat{\mX}(t)\big\rVert_F^2\right]\nonumber\\
    &= \omega^2\big\lVert \mZ(t{+}1)-\hat{\mZ}(t{+}1)\big\rVert_F^2,
\end{align}
where according to Jensen's inequality, 
and~\eqref{eq:acg-proof-2-1} we have
\begin{align}\label{eq:acg-proof-2-2}
\bbE \big\lVert \mZ(t{+}1)-\hat{\mZ}{(t+2)}\big\rVert_F \leq \omega\big\lVert \mZ(t{+}1)-\hat{\mZ}(t{+}1)\big\rVert_F.
\end{align}
Hence, we need to bound $\big\lVert \mZ(t{+}1)-\hat{\mZ}(t{+}1)\big\rVert_F$, as follows:
\begin{align}\label{eq:acg-proof-2-3}
& \big\lVert \mZ(t{+}1)-\hat{\mZ}(t{+}1)\big\rVert_F\nonumber\\ \overset{\eqref{eq:z-update}}&{=} \big\lVert [\mT_1\otimes\mI] \mZ(t) + \gamma[\mT_2\otimes\left(\mW{-}\mI\right)]\hat{\mZ}(t{+}1) - \hat{\mZ}(t{+}1)\big\rVert_F \nonumber\\
&= \big\lVert [\mI_{2n}+\gamma\mT_2\otimes\left(\mI-\mW\right)]\left(\mZ(t)-\hat{\mZ}(t{+}1)\right)\nonumber\\
&+ [\mT_3\otimes\mI_n + \gamma T_2 \otimes\left(\mW{-}\mI_n\right)] \left(\mZ(t)-\overline{\mZ}\right)\big\rVert_F\nonumber\\
\overset{\eqref{eq:two-by-two-matrices}}&{\leq} (1+\gamma\beta\kappa_2)\big\lVert\mZ(t)-\hat{\mZ}(t{+}1)\big\rVert_F\nonumber\\
&+ (\kappa_3+\gamma\beta\kappa_2) \big\lVert \mZ(t)-\overline{\mZ}\big\rVert_F.
\end{align}

Based on~\eqref{eq:acg-proof-1},~\eqref{eq:acg-proof-2-1}, and~\eqref{eq:acg-proof-2-2}, we have:
\begin{align*}
\bbE\mcR_{z}(t{+}1) &\leq 2 \lambda^{t{+}1} \big\lVert\mZ{(0)}\big\rVert_F + \gamma\beta\kappa_2C\sum\limits_{s=0}^{t}s\lambda^s \,\mcU_{z}{(t{-}s)},\\
\bbE\,\mcU_{z}(t{+}1) &\leq \omega(1{+}\gamma\beta\kappa_2) \,\mcU_{z}(t)+\omega(\kappa_3{+}\gamma\beta\kappa_2)\,\mcR_{z}(t).
\end{align*}
Let \mbox{$\nu  = \omega(\kappa_3+\gamma\beta\kappa_2)$}, \mbox{$\nu '=\omega(1+\gamma\beta\kappa_2)$}, where \mbox{$\nu \geq\nu'$}. We now by induction show that for 
\begin{align}
    \omega \leq \frac{1}{2(\kappa_3+\gamma\beta\kappa_2)(\lambda^{-\frac{1}{2}} + \gamma \beta\kappa_2C\lambda^{-1}(1-\lambda^\frac{1}{2})^{-2})},
\end{align}
$\mcU_{z}(t)$ satisfies the following inequality:
\begin{align}\label{eq:diff-z-z-hat}
    \mcU_{z}(t) \leq \xi_0\lambda^{t/2},
\end{align}
where $\xi_0 = 4\lambda^{-\frac{1}{2}}\nu \big\lVert\mZ{(0)}\big\rVert_F$. First, one can check~\eqref{eq:diff-z-z-hat} holds for $t=0$. Furthermore,
\begin{align}
&\bbE\,\mcU_{z}(t{+}1)\nonumber\\
&\leq \nu\,\mcU_{z}(t) + 2\nu \lambda^{t}\big\lVert\mZ{(0)}\big\rVert_F + \gamma \nu\beta\kappa_2C\sum\limits_{s=0}^{t-1}s\lambda^s\mcU_{z}{(t{-}s{-}1)}\nonumber\\
&\leq \nu\,\xi_0\lambda^{\frac{t}{2}} + 2\nu \lambda^{t}\big\lVert\mZ{(0)}\big\rVert_F + \gamma\nu \beta\kappa_2C\xi_0\sum\limits_{s=0}^{t-1}s\lambda^s\lambda^{\frac{t{-}s{-}1}{2}}\nonumber\\
&\leq \left(\frac{2\nu}{\lambda^{\frac{1}{2}}}\big\lVert\mZ{(0)}\big\rVert_F + \frac{\nu\xi_0}{\lambda^{\frac{1}{2}}} + \frac{\gamma\nu \beta\kappa_2C\xi_0}{\lambda(1-\lambda^{\frac{1}{2}})^2}\right)\lambda^{\frac{t+1}{2}}\nonumber\\
&\leq \left(\frac{\xi_0}{2} + \frac{\xi_0}{2}\right)\lambda^{\frac{t+1}{2}} \leq \xi_0\lambda^{\frac{t+1}{2}},
\end{align}
where we used $\sum\limits_{s=0}^{\infty}s\lambda^{\frac{s}{2}} \leq \big(1-\lambda^\frac{1}{2}\big)^{-2}$, using its corresponding generating function. We then bound $\mcR_{z}(t)$:
\begin{align}\label{eq:bound-r-t}
    \bbE\mcR_{z}(t) \overset{}&{\leq} 2 \lambda^{t} \big\lVert\mZ{(0)}\big\rVert_F + \gamma \beta\kappa_2C\sum\limits_{s=0}^{t-1}s\lambda^s \,\mcU_{z}{(t{-}s{-}1)}\nonumber\\
    &\leq \frac{\xi_0}{2\nu}\lambda^{t{-}1} + \frac{\gamma \beta\kappa_2C}{\lambda^\frac{1}{2}}\sum\limits_{s=0}^{t{-}1}s\lambda^s \,\mcU_{z}{(t{-}s{-}1)}\nonumber\\
    &\leq \underbrace{\xi_0\left(\frac{\lambda^{\frac{t}{2}}}{2\nu\lambda} + \frac{\gamma\beta\kappa_2C}{\lambda\big(1-\lambda^{\frac{1}{2}}\big)^2}\right)}_{C_0\text{: constant}}\,\lambda^{\frac{t}{2}} = C_0\,\lambda^{\frac{t}{2}}.
\end{align}
We moreover know that
\begin{align}
\sqrt{\lambda} \leq \sqrt{1-\frac{\sqrt{\gamma}}{5n}+\frac{\gamma}{100n^2}} = 1-\frac{\sqrt{\gamma}}{10n} = \tilde{\lambda},
\end{align}
then, $\mcR_{z}(t) \leq C_0\tilde{\lambda}^{t}$, which concludes the proof.
\end{proof}

\section{Numerical Experiments}

\label{sec:experiments}

Here, we present a set of numerical results to illustrate the communication advantages of our method. We consider the decentralized average consensus problem for a set of $n$ agents with vectors of size $d=150$. We consider two classes of networks with slow mixing times, path and ring, as well as operator \mbox{$\mathrm{qsgd}_{k}$} for message compression.

Figure~\ref{fig:gossip-qsgd} presents two different experiments. First, we compare the performance of \textit{CG} versus \textit{SCG} given the same quantization operators, \mbox{$\mathrm{qsgd}_{5}$}. We consider path graphs with size $n$ varying from $10$ to $200$, and given a random set of initial parameters, consider the number of iterations $t$ for each algorithm to reach an \mbox{$\epsilon$-consensus}, i.e., \mbox{$\Psi_{x}(t)\leq \epsilon$}, for $\epsilon=10^{{-}4}$.  We run each algorithm $10$ times and average the results. We apply a grid line search for the optimal $\gamma$ in each case. As shown in Fig.~\ref{fig:gossip-qsgd-n-iteration}, our algorithm requires a fewer number of iterations to reach consensus compared to \textit{CG}.

We furthermore provide a comparison between \textit{EG}, \textit{SEG}, \textit{CG}, and \textit{SCG} in Fig.~\ref{fig:gossip-qsgd-iteration-bit}. We consider a ring graph with \mbox{$n=120$}, and random parameters with dimension \mbox{$d=150$}. We show the decay of \mbox{$\Psi_{x}(t)$} based on the number of communications (left) as well as the number of transmitted bits (right). Figure~\ref{fig:gossip-qsgd-iteration-bit} shows that \textit{SCG} requires approximately the same number of communication rounds as \textit{SEG}, with only $10 \%$ of bits transmitted to reach the same accuracy $\epsilon=10^{{-}4}$.

We end this section with an example that explains the role of step-size $\gamma$ in the trade-off between the convergence rate and compression feasibility.
Similar to~Fig.~\ref{fig:gossip-qsgd-n-iteration}, we consider the number of iterations for our algorithm to reach an \mbox{$\epsilon$-convergence} for an average consensus problem with \mbox{$d=100$}, over path networks with varying size $n$ with quantizer \mbox{$\mathrm{qsgd}_{3}$}. We consider a range of step-sizes \mbox{$\gamma\in[0.001,0.025]$}, and for each one, we run our algorithm for different choices of $n$.  As shown in Figure~\ref{fig:gossip-qsgd-feasibility}, given a fixed quantization ratio, $\gamma$ imposes a trade-off between the convergence rate versus the feasibility of the consensus for $\omega$. Hence, a better rate requires a larger $\gamma$, which requires a smaller compression ratio $\omega$, while for a larger $\omega$, we need to decrease $\gamma$, which slows down the convergence rate.

\begin{figure}[!t]
    \centering
    \includegraphics[width=0.95\linewidth]{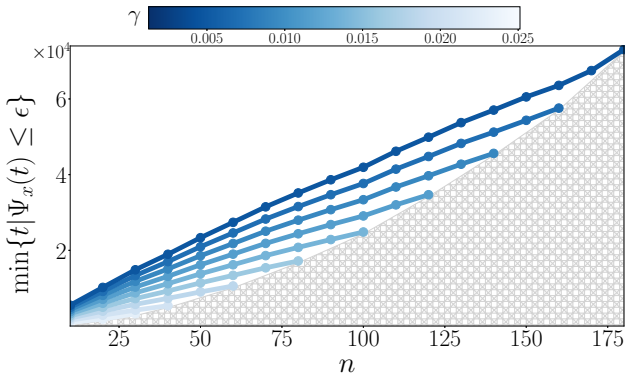}
    \caption{\textbf{Stepsize impact on convergence rate:} Each simulation is the average of 10 runs. Each point shows the number of rounds required for \textit{SCG} to reach an \mbox{$\epsilon$-consensus} ($\epsilon=10^{{-}3}$) for an average consensus problem with \mbox{$d=100$} over a path network with $n$ agents using the compression operator \mbox{$\mathrm{qsgd}_{3}$}. Each line associates with a fixed stepsize $\gamma$, and the hatched area shows the set of $n$, which for \textit{SCG} diverges given the corresponding $\gamma$.}
    \label{fig:gossip-qsgd-feasibility}
\end{figure}

\section{Conclusions}
\label{sec:conclusion}
In this work, we proposed a scalable communication-efficient algorithm for the problem of decentralized average consensus. Given a large enough compression ratio, we showed that agents can communicate compressed messages yet reach consensus with a linear rate that depends linearly on the number of agents in the network. We further presented numerical results to illustrate our theoretical studies. Future work should investigate the combined effect of communication efficiency and scalability in decentralized problems like optimization and inference using the proposed consensus technique. The impact of byzantine agents and other variations of the consensus problem remain as future work.

\bibliographystyle{IEEEbib}
\bibliography{ref}



\end{document}